\documentclass[11pt]{amsart}
\theoremstyle{plain}
\newtheorem{thm}{Theorem}[section]
\newtheorem{theorem}[thm]{Theorem}

\newtheorem{lemma}[thm]{Lemma}

\newtheorem{proposition}[thm]{Proposition}
\theoremstyle{definition}

\newtheorem{definition}[thm]{Definition}

\numberwithin{equation}{section}
\newcommand{\C}{{\mathbb C}}

\newcommand{\BP}{{\mathbb P}}
\newcommand{\Q}{{\mathbb Q}}
\newcommand{\R}{{\mathbb R}}

\newcommand{\Z}{{\mathbb Z}}

\title [Primitive automorphisms]{Primitive automorphisms of a simple abelian variety}

\author{Keiji Oguiso}

\address{Mathematical Sciences, the University of Tokyo, Meguro Komaba 3-8-1, Tokyo, Japan and Korea Institute for Advanced Study, Hoegiro 87, Seoul, 
133-722, Korea}
\email{oguiso@ms.u-tokyo.ac.jp}

\thanks{The author is supported by JSPS Grant-in-Aid (S) No 25220701, JSPS Grant-in-Aid (S) 15H05738, JSPS Grant-in-Aid (A) 16H02141, JSPS Grant-in-Aid (B) 15H03611, and by KIAS Scholar Program.}

\begin{document}

\maketitle

\begin{abstract} We shall prove that an automorphism of a simple abelian variety is primitive if and only if it is of infinite order. 
\end{abstract}

\section{Introduction}

This note provides a supplementary result (Theorem \ref{thm1}) of my talk at the sixty-first Algebra Symposium of Mathematical Society of Japan, held at Saga University on September 7--10, 2016. My talk there was based on my previous \cite{Og16-2}. 

Throughout this note, the base field is assumed to be the complex number field 
$\C$. Let $M$ be a smooth projective variety of dimension $m \ge 2$ and $f \in {\rm Bir}\, (M)$. 

$f$ is said to be {\it imprimitive} if there are a smooth projective variety 
$B$ with $0 < \dim\, B < m$ and a dominant rational map $\pi : M \dasharrow B$ with connected fibers such that $\pi$ is $f$-equivariant, i.e., there is $f_B \in {\rm Bir}\, (B)$ satisfying $\pi \circ f = f_B \circ \pi$. As $\pi$ is just a rational dominant map, smoothness assumption of $B$ is harmless by Hironaka resolution of singularities (\cite{Hi64}). We say that $f$ is {\it primitive} if it is not imprimitive. 

The notion of primitivity is introduced by De-Qi Zhang \cite{Zh09}. Note that if $f$ is primitive, then ${\rm ord}\, (f) = \infty$. Indeed, otherwise, the invariant field $\C(M)^{f^*}$ is of the same transcendental degree $m$ as the rational function field $\C(M)$. Thus we have $\varphi \in \C(M)^{f} \setminus \C$ as $m \ge 1$. Then the Stein factorization of $\varphi : M \dasharrow \BP^1$ is $f$-equivariant. $f$ is then imprimitive as $m \ge 2$.

Assume that $f \in {\rm Aut}\, (M)$. The {\it topological entropy} $h_{{\rm top}}(f)$ of $f$ is a fundamental quantity measuring the complexity of the orbit behaviour under $f^n$ ($n \ge 0$). Let $r_{p}$ be the spectral radius of $f^{*}|H^{p, p}(M)$. Then, by Gromov-Yomdin's theorem, $h_{{\rm top}}(f)$ satisfies 
$$0 \le h_{{\rm top}}(f) = \log {\rm max}_{0 \le p \le m} r_p(f)\,\, $$ 
In this note, it is harmless to regard this formula as the definition of $h_{{\rm top}}(f)$ (See eg. \cite{Og15} and references therein for details).  

The aim of this note is to remark the following:

\begin{theorem}\label{thm1} Let $A$ be a simple abelian variety of dimension $m \ge 2$ and $f \in {\rm Aut}\, (A)$. Then $f$ is primitive if and only if 
${\rm ord}\, (f) = \infty$. In particular, the translation automorphism $t_a$ ($a \in A$) defined by $x \mapsto x +a$ is primitive if $a$ is a non-torsion point of $A$ with fixed zero. Moreover, if in addition $A$ is of CM type, then $A$ admits a primitive automorphism of positive entropy, possibly after replacing $A$ by an isogeny. 
\end{theorem}

Here and hereafter, an abelian variety $A = \C^m/\Lambda$ is said to be {\it simple} if $A$ has no abelian subvariety $B$ such that $0 < \dim\, B < \dim\, A$. A simple abelian variety $A$ is called {\it of CM type} if the endomorphism ring $E := {\rm End}_{{\rm group}}(A) \otimes \Q$ is a CM field with $[E: \Q] = 2\dim\, A$. By definition, a field $E$ is a {\it CM field} if $E$ is a totally imaginary quadratic extension of a totally real number field $K$. Note that if an abelian variety $B$ is isogenous to a simple abelian variety of CM type, then so is $B$ with the same endomorphism ring as $A$. However, ${\rm Aut}_{{\rm group}}\, (A) \not\simeq {\rm Aut}_{{\rm group}}\, (B)$ in general (even for elliptic curves of CM type). 

The "only if" part of Theorem \ref{thm1} is clear as already remarked. Theorem \ref{thm1} is a generalization of our earlier work \cite[Theorem 4.3]{Og16-2}. The last statement of Theorem \ref{thm1} gives an affrimative answer to a question asked by Gongyo at the symposium. 

Our proof is a fairly geometric one based on works due to Amerik-Campana \cite{AC13} and Bianco \cite{Bi16} and is in some sense close to \cite{Og16-3}. 

{\bf Acknowledgement.} I would like to express my thanks to Professors Tomohide Terasoma, Kota Yoshioka and Fumiharu Kato for their invitation to the symposium, Professor Yoshinori Gongyo for his inspiring question there and Professor Akio Tamagawa for his interest in this work and valuable e-mail correspondence. 

\section{Proof of Theorem \ref{thm1}.}

Let $A$ be a simple abelian variety of dimension $m \ge 2$ and $f \in {\rm Aut}\, (A)$ such that ${\rm ord}\, (f) = \infty$. We first show that $f$ is primitive.  

The following two well-known propositions will be frequently used: 

\begin{proposition}\label{prop21}
Let $V$ be a subvariety of $A$ such that $\dim\, V < m = \dim\, A$ and $\tilde{V}$ is a Hironaka resolution of $V$. Then $\tilde{V}$ is of general type. 
\end{proposition}
\begin{proof} See \cite[Corollary 10.10]{Ue75}.
\end{proof}

\begin{proposition}\label{prop22}
Let $M$ be a smooth projective variety of general type defined over a field $k$ of characteristic $0$. Then the birational automorphism group ${\rm Bir}\, (M/k)$ of $M$ over $k$ is a finite group
\end{proposition}
\begin{proof} By the Lefschetz principle, we may reduce to \cite[Corollary 14.3]{Ue75}.
\end{proof}

\begin{lemma}\label{lem21} Let $P$ be a very general closed point of $A$. Then the $\langle f \rangle$-orbit $\{f^n(P)\, |\, n \in \Z\}$ of $P$ is Zariski dense in $A$. 
\end{lemma}

\begin{proof} As $P$ is very general, $f^n$ is defined at $P$ for all $n \in \Z$. By \cite[Th\'eor\`eme 4.1]{AC13}, there is a smooth projective variety $B$ and a dominant rational map $\rho : A \dasharrow B$ such that $\rho \circ f = \rho$ and $\rho^{-1}(\rho(P))$ is the Zariski closure of $\langle f \rangle$-orbit of $P$. It suffices to show that $\dim B = 0$. {\it In what follows, assume to the contray that $\dim\, B > 0$, we derive a contradiction.} 

Let $\eta \in B$ be the generic point in the sense of scheme and $A_{\eta}$ be the fiber over $\eta$. Then by Proposition \ref{prop21} and specialization, a Hironaka resolution of each irreducible component of $A_{\eta}$ is of general type over $\C(B)$. By $\rho \circ f = \rho$, $f$ faithfully acts on $A_{\eta}$ over $\C(B)$. Thus, by Proposition \ref{prop22}, $f^n = id$ on $A_{\eta}$ for some positive integer $n$. Thus $f^n = id$ on $A$, as the generic point $\eta_{A}$ of $A$ is in $A_{\eta}$. This contradicts to ${\rm ord}\, f = \infty$. 
\end{proof}

The following general, useful proposition is due to Bianco: 

\begin{proposition}\label{prop23} Let $X$ be a projective variety and $g \in {\rm Bir}\, (X)$. Assume that 
$\pi : X \dasharrow B$ is a $g$-equivariant dominant rational map to a smooth projective variety $B$ with $\dim\, B < \dim\, X$. Assume that a Hironaka resolution $\tilde{X}_b$ of the fiber $X_b$ is of general type for a general closed point $b \in B$. Then for any very general closed point $P \in X$, the $\langle g \rangle$-orbit $\{g^n(P)| n \in \Z\}$ of $P$ is never Zariski dense in $X$.  
\end{proposition}

\begin{proof} See \cite[Section 4]{Bi16}. See also \cite[Remark 2.6]{Og16-3} for a minor clarification. 
\end{proof}
The next proposition completes the first part of Theorem \ref{thm1}: 
\begin{proposition}\label{prop24} Let $A$ be a simple abelian variety of dimension $\ge 2$ and $f$ be an automorphism of $A$ of infinite order. Then $f$ is primitive. 
\end{proposition}

\begin{proof} Let $\pi : A \dasharrow B$ be an $f$-equivariant dominant rational map to a smooth projective variety $B$ with $\dim\, B < \dim\, A$ and with connected fibers. If $\dim\, B > 0$, then by Proposition \ref{prop21}, a Hironaka resolution $\tilde{A}_b$ of the fiber $A_b$ over $b \in B$ is of general type for general $b \in B$. Then, by Proposition \ref{prop23}, the $\langle f \rangle$-orbit of a very general closed point $P \in A$ is not Zariski dense. This contradicts to Lemma \ref{lem21}. Thus $\dim\, B = 0$, i.e., $f$ is primitive. 
\end{proof}

We shall show the last part of Theorem \ref{thm1}. 

Let $A$ be a simple abelian variety of CM type of dimension $m \ge 2$. We write $E := {\rm End}_{{\rm group}}(A) \otimes \Q$. Then by definition, $E$ is a totally imaginary quadratic extension of a totally real number field $K$ with $[K : \Q] = m \ge 2$. First we make $A$ explicit up to isogeny. As $E$ is a totally imaginary field with $[E:\Q] = 2m$, there are exactly $2m$ different complex embeddings $\varphi_i : E \to \C$ ($1 \le i \le 2m$) such that $\varphi_{2m-i} = \overline{\varphi_i}$. Here $-$ is the complex conjugate of $\C$. Note that there are exactly $2^m\cdot m!$ ways of numberings $I$ of the embeddings here. Choosing one such numbering $I$, we consider the embedding:
$$\varphi_I := (\varphi_1, \varphi_2, \cdots , \varphi_m) : E \to \C^m\,\, ;\,\, a \mapsto (\varphi_1(a), \varphi_2(a), \ldots , \varphi_m(a))\,\, .$$
Let $O_{E}$ (resp. $O_K$) be the integral closure of $\Z$ in $E$ (resp. in $K$). Then
$$B_I := \C^m/\varphi_I(O_E)$$
is an abelian variety and $A$ is isogenous to $B_I$ for some numbering $I$ (See eg. \cite[Chapter I, Section 3]{Mi06}). 

From now, we shall prove that the abelian variety $B := B_I$ admits an automorphism of positive entropy. 

\begin{definition}\label{def21} Let $\overline{\Q}$ be the algebraic closure of $\Q$ in $\C$, $\overline{\Z}$ be the integral closure of $\Z$ in 
$\overline{\Q}$ and $\overline{\Z}^{\times}$ be the unit group of the ring $\overline{\Z}$. A real algebraic integer is an element of $\overline{\Z} \cap \R$. A real algebraic integer $\alpha$ is called a {\it Pisot number} if $\alpha > 1$ and $|\alpha'| < 1$ for all Galois conjugates $\alpha' \not= \alpha$ of $\alpha$ over $\Q$. A Pisot number $\alpha$ is called a {\it Pisot unit} if $\alpha \in \overline{\Z}^{\times}$.
\end{definition} 

Then, by \cite[Theorem 5.2.2]{BDGPS92}, we have

\begin{theorem}\label{thm21}
For any real number field $L$, there is a Pisot unit $\alpha \in L$ such that 
$L = \Q(\alpha)$. 
\end{theorem}

As $K$ is (totally) real, there is then a Pisot unit $\alpha$ such that $K = \Q(\alpha)$. Consider the linear automorphism of $\C^m$ defined by:
$$\tilde{f}_{\alpha} : \C^d \to \C^d\,\, ;\,\, (z_1, z_2, \ldots, z_m) \mapsto (\varphi_1(\alpha)z_1, \varphi_2(\alpha)z_2, \ldots, \varphi_m(\alpha)z_m)\,\, .$$
As $\alpha$ is a unit in $O_K$ (hence in $O_E$), so are $\varphi_i(\alpha)$ in $\varphi_i(O_E)$. Thus $\tilde{f}_{\alpha}(\varphi_I(O_E)) = \varphi_I(O_E)$ by the definition of $\varphi_I$. Hence $\tilde{f}_{\alpha}$ descends to an automorphism $f_{\alpha}$ of $B$. We set $f := f_{\alpha}$. 

As $K$ is totally real, regardless of $I$, we have 
$$\{\varphi_i(\alpha)\,|\, 1 \le i \le m \} = \{\alpha := \alpha_1, \alpha_2, \ldots , \alpha_m \}\,\, .$$
Here the right hand side is the set of all Galois conjugates of $\alpha$ over $\Q$. By the construction of $f$ from $\tilde{f}_{\alpha}$, the left hand side set also coincides with the set of eigenvalues of $f_*|H^0(B, \Omega_B^1)^{*}$, and therefore, coincides with the set of eigenvalues of $f^*|H^0(B, \Omega_B^1)$. 
As $B$ is an abelian variety, we have
$$H^{1,1}(B) = H^0(B, \Omega_B^1) \otimes  \overline{H^0(B, \Omega_B^1)}\,\, .$$Here $\overline{H^0(B, \Omega_B^1)}$ is the complex conjugate of $H^0(B, \Omega_B^1) \subset H^1(B, \Z) \otimes \C$. As $\alpha$ is real, it follows that $\alpha^2$ is an eigenvalue of the action of $f$ on $H^{1,1}(B)$. Hence 
$$h_{{\rm top}}(f) \ge r_1(f) \ge \alpha^2 > 1\,\, .$$ 
Here the last inequality follows from the fact that $\alpha > 1$. Thus $f$ is of postive entropy. In particular, ${\rm ord}\, (f) = \infty$. Therefore, $f$ is primitive as well by the first part of Theorem \ref{thm1}. This completes the proof of Theorem \ref{thm1}.

\end{document}